\documentclass[twoside,12pt]{article}

\usepackage{amsmath,amssymb,amsbsy,amsfonts,amsthm,latexsym,amsopn,amstext,amsxtra,euscript,amscd}

\makeatletter
\let\@fnsymbol\@arabic
\makeatother

\begin{document}

\def\halmos{\rule{6pt}{6pt}}
\def\ve{\varepsilon}

\newtheorem{theorem}{Theorem}
\newtheorem{lemma}{Lemma}
\newtheorem{claim}[theorem]{Claim}
\newtheorem{cor}[theorem]{Corollary}
\newtheorem{prop}[theorem]{Proposition}
\newtheorem{definition}{Definition}
\newtheorem{question}{Question}
\newtheorem{rem}[theorem]{Remark}
\newtheorem{conj}{Conjecture}

\pagestyle{myheadings}
\markboth{\centerline{G. Soydan, L. N\'emeth, L. Szalay}}{On the Diophantine equation $\sum_{j=1}^kjF_j^p=F_n^q$}

\title{\bf On the Diophantine equation $\sum_{j=1}^kjF_j^p=F_n^q$}

\author{
{\sc G\"okhan~Soydan\footnote{Department of Mathematics, Uluda\u{g} University, G\"{o}r\"{u}kle Kamp\"{u}s, 16059 Bursa, Turkey. \textit{gsoydan@uludag.edu.tr}}, 
L\'aszl\'o~N\'emeth\footnote{University of Sopron, Institute of Mathematics, H-9400, Sopron, Bajcsy-Zs. utca 4., Hungary. \textit{nemeth.laszlo@uni-sopron.hu}},
 L\'aszl\'o~Szalay}\footnote{University of Sopron, Institute of Mathematics, H-9400, Sopron, Bajcsy-Zs. utca 4., Hungary; J. Selye University, Institute of Mathematics and Informatics, 94501, Kom\'arno, Bratislavska cesta 3322, Slovakia. \textit{szalay.laszlo@uni-sopron.hu}}} 

\date{\today}
\date{}

\maketitle

\begin{abstract}
Let $F_n$ denote the $n^{th}$ term of the Fibonacci sequence. In this paper, we investigate the Diophantine equation 
$F_1^p+2F_2^p+\cdots+kF_{k}^p=F_{n}^q$
in the positive integers $k$ and $n$, where $p$ and $q$ are given positive integers. A complete solution is given if the exponents are included in the set $\{1,2\}$. Based on the specific cases we could solve, and a computer search with $p,q,k\le100$ we conjecture that beside the trivial solutions only $F_8=F_1+2F_2+3F_3+4F_4$, $F_4^2=F_1+2F_2+3F_3$, and $F_4^3=F_1^3+2F_2^3+3F_3^3$ satisfy the title equation.\\[1mm]
\end{abstract}
\bigskip



\section{Introduction}

As usual, let $(F_n)_{n\ge 0}$  and $(L_n)_{n\ge 0}$ denote the
sequences of Fibonacci and Lucas numbers, respectively, given by
the initial values $F_0=0,~F_1=1,~L_0=2,~L_1=1$, and by the recurrence relations
\begin{equation}\label{erra}
F_{n+2}=F_{n+1}+F_n\qquad {\text{\rm and}}\qquad
L_{n+2}=L_{n+1}+L_{n}\qquad {\text{\rm for~all}}~n\ge 0,
\end{equation}
respectively.  Putting $\alpha=(1+\sqrt{5})/2$ and
$\beta=(1-\sqrt{5})/2=-1/\alpha$ for the two roots of the common
characteristic equation $x^2-x-1=0$ of the two sequences, the formulae 
\begin{equation*}\label{F0} F_n=\frac{\alpha^n-\beta^n}{\alpha-\beta}\qquad
\qquad{\text{\rm and}}\qquad\qquad L_n=\alpha^n+\beta^n
\end{equation*}
hold for all $n\ge 0$. These numbers are well-known for possessing amazing and wonderful properties (consult, for instance, \cite{V} and \cite{T} together with their very rich annotated bibliography for history and additional references). Observing
\begin{eqnarray*}\label{inisols}
F_1 & = & F_2,  \\ 
F_1+2F_2 &=& F_4,  \\
F_1+2F_2+3F_3& =& F_4^2, \\
F_1+2F_2+3F_3+4F_4&=&F_8,
\end{eqnarray*}
the question arises naturally: is there any rule for $F_1+2F_2+3F_3+\cdots +kF_k$? We study this question more generally, according to the title equation.
 
Diophantine equations among the terms of Fibonacci numbers have a very extensive literature. Here we quote a few results that partially motivated us.

By the defining equality (\ref{erra}) of the Fibonacci numbers and the identity $F_{n}^{2}+F_{n+1}^{2}=F_{2n+1}$ (Lemma 1.8), we see that $F_{n}^{s}+F_{n+1}^{s}$ ($n\geq0$) is a Fibonacci number for $s\in\{1,2\}$. For larger $s$ Marques and Togb\'{e} \cite{MT} proved in 2010 that if $F_{n}^{s}+F_{n+1}^{s}$ is a Fibonacci number for all sufficiently large $n$ then $s=1$ or $2$. Next year Luca and Oyono \cite{LO} completed the solution of the question by showing that apart from $F_1^s+F_2^s=F_3$ there is  no solution $s\geq3$ to the equation $F_{n}^{s}+F_{n+1}^{s}=F_m$.

Let $l,s_{1},...,s_{l},a_{1},...,a_{l}$ be integers with  $l\geq1$ and $s_{j}\geq1$. Suppose that there exists $1\leq t \leq l$ such that $a_{t}\neq0$ and $s_{t}>s_{j}$, for all $j\neq t$. Chaves, Marques and Togb\'{e} \cite{CMT}, showed that if either $s_{t}$ is even or $a_{t}$ is not a positive power of $5$, then the sum
\begin{equation*}
a_{1}F_{n+1}^{s_{1}}+a_{2}F_{n+2}^{s_{2}}+...+a_{l}F_{n+l}^{s_{l}}
\end{equation*}
does not belong to the Fibonacci sequence for all sufficiently large $n$.

A balancing problem having similar flavor has been considered by Behera et al.~\cite{B}. They studied the equation
\begin{equation}\label{march21}
F_1^p+F_2^p+\cdots+F_{k-1}^p=F_{k+1}^q+\cdots+F_{k+r}^q,
\end{equation} 
and solved it for the cases $(p,q)=(2,1),(3,1),(3,2)$, and for $2\le p\le q$ by showing the non-existence of any solution. Further the authors conjectured that only the quadruple $(k, r, p, q) = (4, 3, 8, 2)$ of positive integers satisfies (\ref{march21}). The conjecture was completely justified by Alvarado et al.~\cite{A}.
Note that if $(p,q)=(1,1)$ we obtain the problem of sequence balancing numbers handled by Panda \cite{PA}. 

Recalling the formulae $F_1+F_2+\cdots+F_{k}=F_{k+2}-1$ and $F_1^{2}+F_2^{2}+\cdots+F_{k}^{2}=F_{k}F_{k+1}$, it is obvious that the problems
\begin{equation*}
F_1+F_2+\cdots+F_{k}=F_{n}^q, \quad {\rm and}\quad F_1^{2}+F_2^{2}+\cdots+F_{k}^{2}=F_{n}^{q}
\end{equation*}
are rather simple. Indeed, the equations above lead to the lightsome ones
\begin{equation*}\label{F9}
F_{k+2}-1=F_n^q,\quad\quad F_kF_{k+1}=F_n^q.
\end{equation*}
However the equation $F_1^{p}+F_2^{p}+\cdots+F_{k}^{p}=F_{n}^{q}$ might be taken an interest if $p\ge3$. 

The last motivation of our examination was the Diophantine equation
\begin{equation}\label{F11}
x^2+2(x+1)^2+\cdots +n(x+n-1)^2=y^2
\end{equation}
to determine the values of $n$ for which it has finitely or infinitely many positive integer solutions $(x,y)$ (see Wulczyn \cite{W}, and for details, see also \cite{AA}). For variations of the equation \eqref{F11}, we refer the reader to 
\cite{G}.

In this paper, we investigate the Diophantine equation
\begin{equation}\label{main} 
F_1^p+2F_2^p+\cdots+kF_{k}^p=F_{n}^q
\end{equation}
in the positive integers $k$ and $n$, where $p$ and $q$ are fixed positive integers. 
We consider
$$
F_1^p=1=F_1^q=F_2^q,\qquad{\rm and}\qquad F_1^p+2F_2^p=3=F_4
$$
as trivial solutions to (\ref{main}). 
We have the following conjecture based upon the specific cases we could solve, and a computer search with $p,q,k\le100$.
\begin{conj}
The non-trivial solutions to (\ref{main}) are only
\begin{eqnarray*}
F_4^2&=&\,\,9\,=F_1+2F_2+3F_3, \\
F_8&=&21=F_1+2F_2+3F_3+4F_4, \\
F_4^3&=&27=F_1^3+2F_2^3+3F_3^3.
\end{eqnarray*}
\end{conj}
This work handles the particular cases $p,q\in\{1,2\}$ (hence the first two solutions above will be obtained), the precise results proved are described as follow. 

\begin{theorem} \label{th1} 
If 
\begin{equation}\label{eq1}
F_1+2F_2+\cdots+kF_{k}=F_{n},
\end{equation}
then
$(k,n)=(1,1),~(1,2),~(2,4),~(4,8)$, among them only the last one is non-trivial solution.
\end{theorem}

\begin{theorem} \label{th2} 
The Diophantine equation
\begin{equation}\label{eq2}
F_1^2+2F_2^2+\cdots+kF_{k}^2=F_{n}^2
\end{equation}
possesses only the trivial solutions
$(k,n)=(1,1),~(1,2)$.
\end{theorem}

\begin{theorem} \label{th3} 
If 
\begin{equation}\label{eq3}
F_1+2F_2+\cdots+kF_{k}=F_{n}^2,
\end{equation}
then
$(k,n)=(1,1),~(1,2),~(3,4)$, among them only the last one is non-trivial solution.
\end{theorem}

\begin{theorem} \label{th4} 
The Diophantine equation
\begin{equation}\label{eq4}
F_1^2+2F_2^2+\cdots+kF_{k}^2=F_{n}
\end{equation}
possesses only the trivial solutions
$(k,n)=(1,1),~(1,2),~(2,4)$.
\end{theorem}

\section{Lemmata}
In this section, we present the lemmata that are needed in the proofs of the theorems. The first lemma is a collection of a few well-known results, we state them without proof, and in the proof of the theorems sometimes we do not refer to them.
\begin{lemma}\label{l1}  Let $k$ and $n$ be arbitrary integers.
	\begin{enumerate}
		\item $\sum_{j=1}^kjF_j=kF_{k+2}-F_{k+3}+2$.
		\item $\sum_{j=1}^kjF_j^2=F_k(kF_{k+1}-F_k)+\tau$, where $\tau=0$ if $k$ is even, and $\tau=1$ otherwise.
		\item For $k\ge0$ we have $F_{-k}=(-1)^{k+1}F_k$, further $L_{-k}=(-1)^kL_k$ (extension of the sequences for negative subscripts).
		\item $\gcd(F_k,F_n)=F_{\gcd(k,n)}$.
		\item $\gcd(F_k,L_n)=1$ or $2$ or $L_{\gcd(k,n)}$.
		\item $F_k\mid F_n$ if and only if $k\mid n$.
		\item $F_{k+1}F_n-F_kF_{n+1}=(-1)^{n+1}F_{k-n}$ (d'~Ocagne's identity).
		\item $F_{k+n}=F_{k}F_{n+1}+F_{k-1}F_n$.
		\item $F_{2k}=F_kL_k$.
		\item $F_{k+n}^2-F_{k-n}^2=F_{2k}L_{2n}$. 
	\end{enumerate}
\end{lemma}




\begin{lemma}\label{lfp}
\begin{equation*}\label{elagf-1}
F_k-1 =
\begin{cases} F_{(k+2)/2}L_{(k-2)/2},\quad {\rm if} &\mbox{$k\equiv0\pmod{4}$}  \\ 
F_{(k-1)/2}L_{(k+1)/2},\quad{\rm if} &\mbox{$k\equiv1\pmod{4}$} \\ 
F_{(k-2)/2}L_{(k+2)/2},\quad{\rm if} &\mbox{$k\equiv2\pmod{4}$} \\ 
F_{(k+1)/2}L_{(k-1)/2},\quad{\rm if} &\mbox{$k\equiv3\pmod{4}$} 
\end{cases};
\end{equation*}
\begin{equation*}\label{elagf+1}
F_k+1 =
\begin{cases} F_{(k-2)/2}L_{(k+2)/2},\quad {\rm if} &\mbox{$k\equiv0\pmod{4}$}  \\ 
F_{(k+1)/2}L_{(k-1)/2},\quad{\rm if} &\mbox{$k\equiv1\pmod{4}$} \\ 
F_{(k+2)/2}L_{(k-2)/2},\quad{\rm if} &\mbox{$k\equiv2\pmod{4}$} \\ 
F_{(k-1)/2}L_{(k+1)/2},\quad{\rm if} &\mbox{$k\equiv3\pmod{4}$} 
\end{cases}.
\end{equation*}
\end{lemma}

\begin{proof}
See, for instance, \cite{LSz} and \cite{PP}.
\end{proof}

\begin{lemma}\label{ppp}
\begin{equation*}\label{p-1}
F_k^2-1 =
\begin{cases} F_{k-1}F_{k+1},\quad {\rm if} &\mbox{$k\equiv1\pmod{2}$}  \\ 
F_{k-2}F_{k+2},\quad{\rm if} &\mbox{$k\equiv0\pmod{2}$} 
\end{cases};
\end{equation*}
\begin{equation*}\label{p+1}
F_k^2+1 =
\begin{cases} F_{k-1}F_{k+1},\quad {\rm if} &\mbox{$k\equiv0\pmod{2}$}  \\ 
F_{k-2}F_{k+2},\quad{\rm if} &\mbox{$k\equiv1\pmod{2}$} 
\end{cases}.
\end{equation*}
\end{lemma}

\begin{proof}
See Lemma 3 in \cite{P}.
\end{proof}

\begin{lemma}\label{l2}
	If $j\ge4$ is even, then
	$$
	2F_{j-1}^{\varphi(F_j)-1}\equiv F_{j-3}\pmod{F_j}.
	$$
\end{lemma}

\begin{proof}
	Since $\gcd(F_{j-1},F_j)=1$, and $F_{j-1}^{\varphi(F_j)}\equiv 1\pmod{F_j}$, it is sufficient to show that
	$$2\equiv F_{j-3}F_{j-1}\;\pmod{F_j}.$$ 
	But 
	$$
	F_{j-1}F_{j-3}=(F_{j+1}-F_j)F_{j-3}\equiv F_{j+1}F_{j-3}=F_jF_{j-2}+(-1)^{j-2}F_3\equiv2\pmod{F_j}
	$$
	follows from the definition of the Fibonacci numbers, d'~Ocagne's identity (Lemma \ref{l1}.7), and the parity of $j$.
\end{proof}

\begin{lemma}\label{l3}
	If $j\ge3$ is odd, then
	$$
	F_{j-1}^{\varphi(F_j)-1}\equiv F_{j-2}\pmod{F_j}.
	$$
\end{lemma}

\begin{proof}
	Similarly to the proof of the previous lemma, the statement is equivalent to
	$$1\equiv F_{j-2}F_{j-1}\;\pmod{F_j}.$$ 
	And it is easy to see that
	$$
	F_{j-2}F_{j-1}=(F_{j}-F_{j-1})F_{j-1}\equiv -F_{j-1}^2=-(F_{j-2}F_j+(-1)^jF_{-1})\equiv1\pmod{F_j}.
	$$
\end{proof}

\begin{lemma} \label{l:becs1}
	Let $k_0$ be a positive integer, and for $i\in\{0,1\}$ put
	$$
	\delta_{i}=\log_{\alpha}\left(\frac{1+(-1)^{i-1}\left(|\beta|/\alpha\right)^{k_0}}{\sqrt{5}}\right),
	$$
	where $\log_{\alpha}$ is the logarithm
	in base $\alpha=(1+\sqrt{5})/2$. Then for all integers $k\ge k_0$, the two
	inequalities
	\begin{equation*}
	\label{eq:2} \alpha^{k+\delta_0}\le F_k\le\alpha^{k+\delta_1}
	\end{equation*}
	hold.
\end{lemma}

\begin{proof}
	This is a part of Lemma 5 in \cite{LSz}.
\end{proof}

In order to make the application of Lemma \ref{l:becs1} more
convenient, we shall suppose that $k_0\ge1$. Then we have 

\begin{cor}\label{co}
	If $k\ge1$, then 
	$$\alpha^{k-2}\le F_k\le \alpha^{k-1},$$
	and equality holds if and only if $k=2$, and $k=1$, respectively.
\end{cor}
Now, we are ready to justify the theorems.
\section{Proofs}

{\bf Proof of Theorem \ref{th1}.} 

Verifying the cases $k=1,\dots,5$ by hand we found the solutions listed in Theorem \ref{th1}. Put $\kappa=k+2$, and suppose that $\kappa\ge8$. Consequently, $F_{\kappa-3}\ge5$ and $F_{\kappa}\ge21$. If equation (\ref{eq1}) holds, then $n>\kappa$, and then by Lemma \ref{l1}.1 we conclude
\begin{equation}\label{egesz}
k=\frac{F_n+F_{\kappa+1}-2}{F_\kappa}=\frac{F_n+F_{\kappa-1}-2}{F_\kappa}+1\in\mathbb{N}.
\end{equation}
In the sequel, we study the sequence $(F_u)_{u=0}^\infty$ modulo $F_\kappa$ if $\kappa$ is fixed. Note that we indicate a suitable value congruent to $F_u$ modulo $F_\kappa$, not always the smallest non-negative remainders. The  period can be deduced from the range 
$$
\overbrace{0,1,1,2,\dots,F_{\kappa-2},F_{\kappa-1}}^{\kappa},\overbrace{0,F_{\kappa-1},F_{\kappa-1},2F_{\kappa-1},\dots,F_{\kappa-2}F_{\kappa-1},F_{\kappa-1}F_{\kappa-1}}^{\kappa}\,,
$$
of length $2\kappa$ if $\kappa$ is even, since then, by Lemma \ref{l1}.7 we have
$F_{\kappa-1}^2\equiv1\pmod{F_\kappa}$ and then $$F_{\kappa-2}F_{\kappa-1}=(F_{\kappa}-F_{\kappa-1})F_{\kappa-1}\equiv-1\pmod{F_\kappa}.$$ 

In case of odd $\kappa$ we have $F_{\kappa-1}^2\equiv-1\pmod{F_\kappa}$, therefore the length of the period is $4\kappa$ coming from 
\begin{eqnarray*}
&&\overbrace{0,1,1,2,\dots,F_{\kappa-2},F_{\kappa-1}}^{\kappa},\overbrace{0,F_{\kappa-1},F_{\kappa-1},2F_{\kappa-1},\dots,F_{\kappa-2}F_{\kappa-1},F_{\kappa-1}F_{\kappa-1}}^{\kappa},\\
&&\qquad\overbrace{0,-1,-1,-2,\dots,-F_{\kappa-2},-F_{\kappa-1}}^{\kappa},\overbrace{0,-F_{\kappa-1},-F_{\kappa-1},-2F_{\kappa-1},\dots,-F_{\kappa-2}F_{\kappa-1},-F_{\kappa-1}F_{\kappa-1}}^{\kappa}.\\
\end{eqnarray*}
\bigskip
Based on the length of the period we distinguish two cases.

{\bf Case I: $\kappa$ is even.}
Either $F_n\equiv F_j$ or $F_n\equiv F_jF_{\kappa-1}$ modulo $\kappa$ holds for some $j=0,1,\dots,\kappa-1$. 
Hence 
\begin{equation}\label{elag1}
F_n+F_{\kappa-1}-2 \equiv \begin{cases} F_j+F_{\kappa-1}-2,\,{\rm or} &\mbox{}  \\ 
F_jF_{\kappa-1}+F_{\kappa-1}-2 & \mbox{} \end{cases} \pmod{F_\kappa}.
\end{equation}

We will show that none of them is congruent to 0 modulo $F_\kappa$. In the first branch
$$
F_j+F_{\kappa-1}-2\ge F_{\kappa-1}-2\ge11,
$$
further if $j\ne\kappa-1$, then
$$
F_j+F_{\kappa-1}-2\le F_{\kappa-2}+F_{\kappa-1}-2\le F_\kappa-2.
$$
Thus $F_j+F_{\kappa-1}-2\not\equiv0\pmod{F_\kappa}$, hence (\ref{egesz}) does not hold. Assume now, that $j=\kappa-1$. Then, together with the definition of the Fibonacci sequence we have
$$
F_j+F_{\kappa-1}-2=F_{\kappa-1}+(F_{\kappa}-F_{\kappa-2})-2\equiv F_{\kappa-3}-2\pmod{F_\kappa}.
$$
But $3\le F_{\kappa-3}-2<F_\kappa$ contradicts to (\ref{egesz}).

Choosing the second branch of (\ref{elag1}), suppose that $F_{\kappa-1}(F_j+1)-2$ is congruent to 0 modulo $F_\kappa$.
Then
$$
F_{\kappa-1}^{\varphi(F_\kappa)}(F_j+1)\equiv 2F_{\kappa-1}^{\varphi(F_\kappa)-1}\pmod{F_\kappa}.
$$
Subsequently, by Lemma \ref{l2}, it leads to
$$
F_j+1\equiv F_{\kappa-3}\pmod{F_\kappa}.
$$
Since $j=0,1,\dots,\kappa-1$, ($\kappa\ge8$) it follows that $F_j=F_{\kappa-3}-1$, a contradiction.
\bigskip

{\bf Case II: $\kappa$ is odd.}
Now $\kappa\ge9$, and either $F_n\equiv \pm F_j\;\pmod{F_\kappa}$ or $F_n\equiv \pm F_jF_{\kappa-1}\pmod{F_\kappa}$ holds for some $j=0,1,\dots,\kappa-1$. 
Hence 
\begin{equation*}\label{elag2}
F_n+F_{\kappa-1}-2 \equiv \begin{cases} \pm F_j+F_{\kappa-1}-2 &\mbox{}  \\ 
\pm F_jF_{\kappa-1}+F_{\kappa-1}-2 & \mbox{} \end{cases} \pmod{F_\kappa}.
\end{equation*}
\bigskip
First, obviously, if $j\ne\kappa-1$, then
$$
6\le F_{\kappa-3}-2\le\pm F_j+F_{\kappa-1}-2\le F_\kappa-2,
$$
so dividing $\pm F_j+F_{\kappa-1}-2$ by $F_\kappa$, the result is not an integer. If $j=\kappa-1$, then the treatment of the ``$+$'' case coincides the treatment when $\kappa$ was even. The ``$-$''
case leads to $F_n+F_{\kappa-1}-2 \equiv -2\pmod{F_\kappa}$, a contradiction.

Assume now that $F_n+F_{\kappa-1}-2 \equiv\pm F_jF_{\kappa-1}+F_{\kappa-1}-2\pmod{F_\kappa}$. Thus
$F_{\kappa-1}(1\pm F_j)\equiv2\pmod{F_\kappa}$. Multiplying both sides by $F_{\kappa-1}^{\varphi(F_\kappa)-1}$, by Lemma \ref{l3} it gives
$$
1\pm F_j\equiv2 F_{\kappa-2}\pmod{F_\kappa}.
$$
First let $F_j=2F_{\kappa-2}-1$, which leads immediately a contradiction via $0<2F_{\kappa-2}-1=F_{\kappa-1}+F_{\kappa-4}-1<F_\kappa$. If $F_\kappa-F_j+1=2F_{\kappa-2}$, then $F_j=F_{\kappa-3}+1$ follows, a contradiction again. The proof of Theorem \ref{th1} is complete.
\bigskip

{\bf Proof of Theorem \ref{th2}.} 

For the range $k=1,2,\dots,20$ we checked (\ref{eq2}) by hand. From now we assume $k\ge21$.
Based on Lemma \ref{l1}.2, we must distinguish two cases.
\bigskip

{\bf Case I: $k$ is even.} Consider the equation
\begin{equation*}\label{c}
F_k(kF_{k+1}-F_k)=F_n^2.
\end{equation*} 
Trivially, $n>k$. Put $\nu=\gcd(k,n)$.


If $\nu=k$, then $F_k\mid F_n$ by Lemma \ref{l1}.6. Consequently, 
$$
\left(\frac{F_n}{F_k}\right)^2=\frac{kF_{k+1}-F_k}{F_k}=\frac{kF_{k+1}}{F_k}-1
$$
is integer. But $F_k$ and $F_{k+1}$ are coprime, hence $F_k\mid k$, and it results $k\le5$, a contradiction.

Examine the possibility $\nu=k/2$. Put $\kappa=k/2$. Now
$F_\kappa L_\kappa(kF_{k+1}-F_\kappa L_\kappa)=F_n^2$
leads to
$$
\frac{L_\kappa(kF_{k+1}-F_\kappa L_\kappa)}{F_\kappa}=\left(\frac{F_n}{F_\kappa}\right)^2.
$$
This is an equality of integers, which together with $\gcd(F_\kappa,F_{k+1})$ and $\gcd(F_\kappa,L_\kappa)=1,2$ (see Lemma \ref{l1}.5) shows that
$2k/F_\kappa$ is integer. Thus $k\le14$, a contradiction.

Finally, we have $3\le\nu\le k/3$. Since $\gcd(F_k/F_\nu,F_n/F_\nu)=1$, then from the equation
$$
\frac{F_k}{F_\nu}(kF_{k+1}-F_k)=\frac{F_n}{F_\nu}F_n
$$
we conclude
$$
\frac{F_k}{F_\nu}\mid F_n\qquad {\rm and}\qquad \frac{F_n}{F_\nu}\mid kF_{k+1}-F_k.
$$
Subsequently, $F_k\mid F_\nu F_n$ and $F_n\mid F_\nu(kF_{k+1}-F_k)$. Thus $F_k\mid F_\nu^2(kF_{k+1}-F_k)$, and then
$F_k\mid kF_\nu^2$ holds since $\gcd(F_k,F_{k+1})=1$. Applying Corollary \ref{co}, we obtain
$$
\alpha^{k-2}\le F_k\le kF_\nu^2\le k\alpha^{2\nu-2}\le k\alpha^{2/3k-2},
$$
which implies $k<19$.

\bigskip

{\bf Case II: $k$ is odd.} In this part, we follow the idea of the previous case. Recall that $k\ge21$. Now 
$$F_k(kF_{k+1}-F_k)=F_{n-\ve}F_{n+\ve},$$ 
where $\ve=1$ or $2$ depending on the parity of $n$ (see Lemma \ref{ppp}). Clearly, $\gcd(n-\ve,n+\ve)=2$ or 4. Thus $\gcd(F_{n-\ve},F_{n+\ve})=1$ or 3, respectively.

Put $\nu_1=\gcd(k,n-\ve)$ and $\nu_2=\gcd(k,n+\ve)$. Obviously, $\gcd(\nu_1,\nu_2)$ divides $\gcd(n-\ve,n+\ve)$. Hence
$\nu=\gcd(\nu_1,\nu_2)=1$ or 2 or 4, and then $F_\nu=\gcd(F_{\nu_1},F_{\nu_2})=1$ or 3. Thus $F_{\nu_1}F_{\nu_2}\mid F_\nu F_k$. The terms of both the left and right sides of
$$
\frac{F_k}{F_{\nu_1}}(kF_{k+1}-F_k)=\frac{F_{n-\ve}}{F_{\nu_1}}F_{n+\ve} \qquad{\rm and}\qquad
\frac{F_k}{F_{\nu_2}}(kF_{k+1}-F_k)=F_{n-\ve}\frac{F_{n+\ve}}{F_{\nu_2}}
$$
are integers, and
$$
\frac{F_{n-\ve}}{F_{\nu_1}}\mid kF_{k+1}-F_k\;,\qquad \frac{F_k}{F_{\nu_2}}\mid F_{n-\ve}.
$$
Combining them, $F_k\mid F_{\nu_1}F_{\nu_2}(kF_{k+1}-F_k)$ follows, and then $F_k\mid kF_{\nu_1}F_{\nu_2}$. The remaining part of the proof consists of three cases.

Suppose first that $\nu_1=k$, i.e.~$F_k\mid F_{n-\ve}$. Observe, that $n-\ve$ and $n+\ve$ are even, and $k\mid (n-\ve)/2$. Thus $k+1\le (n-\ve)/2+1$, which does not exceed $(n+\ve)/2$. Then
$$
kF_{k+1}>\frac{F_{n-\ve}}{F_{\nu_1}}F_{n+\ve}=\frac{F_{(n-\ve)/2}}{F_{\nu_1}}L_{(n-\ve)/2}F_{n+\ve}\ge L_{(n-\ve)/2}L_{(n+\ve)/2}F_{(n+\ve)/2}\ge L_{(n-\ve)/2}L_{(n+\ve)/2}F_{k+1}.
$$
Simplifying by $F_{k+1}$ we conclude
$$
\frac{n-\ve}{2}\ge k> L_{(n-\ve)/2}L_{(n+\ve)/2},
$$
and we arrived at a contradiction since $21\le k<n$.
Note that the same machinery works when $\nu_2=k$, i.e.~$F_k\mid F_{n+\ve}$.

If none of the two conditions above holds, we can assume $\nu_1\le k/3$ and $\nu_2\le k/3$. Indeed, $k$ is odd, so the largest non-trivial divisor of $k$ is at most $k/3$. The application of Corollary \ref{co} gives
$$
\alpha^{k-2}\le F_k\le kF_{\nu_1}F_{\nu_2}\le k\alpha^{\nu_1-1}\alpha^{\nu_2-1}\le k\alpha^{2k/3-2},
$$
and then $k<19$.

The proof of the theorem is complete.

\bigskip

{\bf Proof of Theorem \ref{th3}.} The proof partially follows the proof of Theorem \ref{th1}. The small cases of (\ref{eq3}) can be verified by hand. 
Suppose $\kappa=k+2\ge9$. Similarly to (\ref{egesz}), we have
\begin{equation}\label{egesz1}
k=\frac{F_n^2+F_{\kappa+1}-2}{F_\kappa}=\frac{F_n^2+F_{\kappa-1}-2}{F_\kappa}+1\in\mathbb{N}.
\end{equation}
Now we study the sequence $(F_u^2)_{u=0}^\infty$ modulo $F_\kappa$, and we again indicate the most suitable values by modulo $F_\kappa$, not always the smallest non-negative remainders. Lemma \ref{l1}.10, together with Lemma \ref{l1}.3 
implies
\begin{equation*}\label{sh}
F_{2\kappa\pm j}^2\equiv F_j^2 \pmod{F_\kappa},
\end{equation*}
where $j=0,1,\dots,\kappa-1$. Hence the period having length $2\kappa$ can be given by 
$$
\overbrace{0,1,1,2^2,\dots,F_{\kappa-2}^2,F_{\kappa-1}^2}^{\kappa},\overbrace{0,F_{\kappa-1}^2,F_{\kappa-2}^2,\dots,2^2,1,1}^{\kappa}\,.
$$
Let us distinguish two cases according to the parity of $\kappa$. 

{\bf Case I: $\kappa$ is even.} Put $\kappa=2\ell$. Again by Lemma \ref{l1}.10, together with $\kappa-i=2\ell-i=\ell+(\ell-i)$ and $i=\ell-(\ell-i)$ admits $F_{\kappa-i}^2\equiv F_i^2 \pmod{F_\kappa}$. It reduces the possibilities to $j=0,1,\dots,\ell$.

If  $j\le\ell-1=(\kappa-2)/2$, then
$$
0<F_j^2+F_{\kappa-1}-2\le F_{(\kappa-2)/2}^2+F_{\kappa-1}-2\le F_{\kappa-2}+F_{\kappa-1}-2<F_\kappa
$$
hold since $F_{(\kappa-2)/2}^2\le F_{(\kappa-2)/2}L_{(\kappa-2)/2}=F_{\kappa-2}$.

Suppose now that $j=\ell=\kappa/2$. Repeating the previous idea we find $F_{\kappa/2}^2+F_{\kappa-1}-2\le F_{\kappa}+F_{\kappa-1}-2<2F_\kappa$. Consequently, $F_{\kappa/2}^2+F_{\kappa-1}-2=F_\kappa$ might be fulfilled. Thus
$F_{\kappa/2}^2-2=F_{\kappa-2}$. Recalling $\kappa=2\ell$ we equivalently obtain
$$
F_\ell^2-1=F_{2\ell-2}+1.
$$ 
Both sides have decomposition described in Lemma \ref{ppp} and Lemma \ref{lfp}, respectively, providing
$$
F_{\ell-1}F_{\ell+1}=F_{\ell-2}L_{\ell}\qquad{\rm or}\qquad F_{\ell-2}F_{\ell+2}=F_{\ell}L_{\ell-2}
$$
if $\ell$ is odd or even, respectively.
Firstly, $F_{\ell-2}\mid F_{\ell-1}F_{\ell+1}$, and then $F_{\ell-2}\mid F_{\ell+1}$ holds only for small $\ell$ values.  Secondly, $F_{\ell}\mid F_{\ell-2}F_{\ell+2}$ contradicts to $\ell\ge5$.

{\bf Case II: $\kappa$ is odd.} Now we have $F_{\kappa-i}^2\equiv -F_{i}^2 \pmod{F_\kappa}$. Indeed, Lemma \ref{l1}.8 admits
$$
F_{\kappa-i}^2=(F_\kappa F_{-i+1}+F_{\kappa-1}F_{-i})^2\equiv F_{\kappa-1}^2F_i^2 \pmod{F_\kappa},
$$
and then Lemma \ref{ppp} justifies the statement. 
It makes possible to split the proof into a few parts.

If $j=(\kappa-1)/2$, then $F_{(\kappa-1)/2}^2+F_{\kappa-1}-2<2F_{\kappa-1}-2<2F_\kappa$. Thus
$F_{(\kappa-1)/2}^2+F_{\kappa-1}-2=F_{\kappa}$ is the only one chance to fulfill (\ref{egesz1}). Then apply Lemma \ref{l:becs1} with $k_0=4\le (\kappa-1)/2$ for $F_{\kappa-2}<F_{(\kappa-1)/2}^2$ to reach a contradiction.

If $j\le(\kappa-3)/2$, then
$$
0<F_j^2+F_{\kappa-1}-2\le F_{(\kappa-3)/2}^2+F_{\kappa-1}-2\le F_{\kappa-3}+F_{\kappa-1}-2<F_\kappa
$$
holds since $F_{(\kappa-3)/2}^2\le F_{(\kappa-3)/2}L_{(\kappa-3)/2}=F_{\kappa-3}$.

Finally, if $(\kappa+1)/2\le j\le \kappa-1$, then 
$$
F_j^2+F_{\kappa-1}-2\equiv -F_{\kappa-j}^2 +F_{\kappa-1}-2 \pmod{F_\kappa},
$$
where $1\le J=\kappa-j\le(\kappa-1)/2$. Thus we need to check the equation
$F_J^2+2=F_{\kappa-1}$, since $-2\le -F_J^2+F_{\kappa-1}-2<F_{\kappa-1}<F_\kappa$. When $J\le(\kappa-3)/2$ holds then
$F_J^2+2<F_{\kappa-3}+2<F_{\kappa-1}$. Lastly, $J=(\kappa-1)/2$ leads to
$F_{(\kappa-1)/2}^2+2=F_{\kappa-1}$, and then to $2=F_{J}\left(L_{J}-F_{J}\right)$.
Obviously, it gives $J\le3$. Thus $\kappa\le7$, a contradiction.

\bigskip

{\bf Proof of Theorem \ref{th4}.} 
The statement for (\ref{eq4}) is obviously true if $k\le12$. So we may  assume $k\ge13$. Let $\tau\in\{0,1\}$.
The formula of the summation, together with Corollary \ref{co} implies
$$
\alpha^{n-2}\le F_n=F_k(kF_{k+1}-F_k)+\tau<kF_kF_{k+1}\le\alpha^{\log_\alpha k+k-1+(k+1)-1},
$$
and then
$$
n-k<k+1+\log_\alpha k<\frac{3}{2}k.
$$
Similarly,
$$
\alpha^{n-1}\ge F_n=F_k(kF_{k+1}-F_k)+\tau> F_k(kF_{k+1}-F_{k+1})=(k-1)F_kF_{k+1}>\alpha^{\log_\alpha (k-1)+k-2+(k+1)-2},
$$
subsequently
$$
n-k>k-2+\log_\alpha(k-1)> k+3, 
$$
that is $n-k\ge k+4$.
Putting together the two estimates it gives
\begin{equation}\label{Dor}
2k+4\le n<\frac{5}{2}k.
\end{equation}

{\bf Case I: $k$ is even.} Clearly, $k+2< n-k-1\le3k/2-2$ holds.
By Lemma \ref{l1}.8, we conclude
$$
F_n=F_{k+1}F_{n-k}+F_kF_{n-k-1}=kF_kF_{k+1}-F_k^2,
$$
and equivalently
$$
F_k(F_k+F_{n-k-1})=F_{k+1}(kF_k-F_{n-k}).
$$
Thus $\gcd(F_k,F_{k+1})=1$ admits $F_{k+1}\mid F_k+F_{n-k-1}$.
The periodicity of $(F_u)_{u=0}^\infty$ modulo $F_{k+1}$ guarantees, together with the bounds on $n-k-1$ that
$$
F_k+F_{n-k-1}\equiv F_k+F_jF_k=F_k(F_j+1)\pmod{F_{k+1}}
$$
holds for some $j=1,2,\dots,k/2-3$.
Consequently, $F_{k+1}\mid F_j+1$, a contradiction.
\bigskip

{\bf Case II: $k$ is odd.} Again $k+2< n-k-1\le3k/2-2$ holds. Lemma \ref{lfp} and Lemma \ref{l1}.8 imply
$$
F_k(kF_{k+1}-F_k)=F_{(n-\ve)/2}L_{(n+\ve)/2}=(F_{k+1}F_{(n-\ve)/2-k}+F_kF_{(n-\ve)/2-k-1})L_{(n+\ve)/2},
$$
where $\ve\in\{\pm1,\pm2\}$ according to the modular property of $n$. It leads to
$$
F_{k+1}(kF_k-F_{(n-\ve)/2-k}L_{(n+\ve)/2})=F_k(F_k+F_{(n-\ve)/2-k-1}L_{(n+\ve)/2}).
$$
Thus $F_k\mid F_{(n-\ve)/2-k}L_{(n+\ve)/2}$.

By (\ref{Dor}) it is obvious that 
$$
k<k+1\le\frac{n-2}{2}\le\frac{n+\ve}{2}\le\frac{n+2}{2}\le \frac{5}{4}k+1<2k,
$$
which exclude $\gcd(k,(n+\ve)/2)=k$. Thus  $\gcd(k,(n+\ve)/2)\le k/3$ since $k$ is odd, subsequently $\gcd(F_k,L_{(n+\ve)/2})\le L_{k/3}$.
On the other hand
$$
\frac{n-\ve}{2}-k\le \frac{n+2}{2}-k\le \frac{1}{4}k+1,
$$
which implies $\gcd(F_k,F_{(n-\ve)/2-k})\le F_{k/4+1}$.

Thus, $F_k\mid F_{(n-\ve)/2-k}L_{(n+\ve)/2}$, together with the previous arguments entails $F_k\le F_{k/4+1}L_{k/3}$, but it leads to a contradiction since $L_{k/3}=F_{k/3-1}+F_{k/3+1}<2F_{k/3+1}$, and the application of Corollary \ref{co} on $F_k\le 2F_{k/4+1}F_{k/3+1}$ returns with $k<9$.

\bigskip

{\bf Acknowledgments.} This paper was partially written when the third author visited the Department of Mathematics of Uluda\u{g} University in Bursa. He would like to thank the Turkish colleagues of the department for the kind hospitality. The first author was supported by the Research Fund of Uluda\u{g} University under project numbers: 2015/23, 2016/9.

\end{document}